\documentclass[11pt]{article}
\usepackage{amssymb,amsfonts,amsmath,latexsym,epsf,tikz,url}
\usepackage[usenames,dvipsnames]{pstricks}
\usepackage{pstricks-add}
\usepackage{epsfig}
\usepackage{pst-grad} % For gradients
\usepackage{pst-plot} % For axes
\usepackage[space]{grffile} % For spaces in paths
\usepackage{etoolbox} % For spaces in paths
\makeatletter % For spaces in paths
\patchcmd\Gread@eps{\@inputcheck#1 }{\@inputcheck"#1"\relax}{}{}
\makeatother

\newtheorem{theorem}{Theorem}[section]

\newtheorem{lemma}[theorem]{Lemma}

\newtheorem{definition}[theorem]{Definition}

\newcommand{\proof}{\noindent{\bf Proof.\ }}
\newcommand{\qed}{\hfill $\square$\medskip}

\begin{document}

\title{On the semitotal dominating sets of graphs}

\author{
Saeid Alikhani$^{}$\footnote{Corresponding author} \and Hassan Zaherifar
}

\date{\today}

\maketitle

\begin{center}
Department of Mathematics, Yazd University, 89195-741, Yazd, Iran\\
{\tt   alikhani@yazd.ac.ir }
\end{center}

\begin{abstract}
A set $D$ of vertices in an isolate-free graph $G$ is a semitotal dominating set of $G$ if $D$ is a dominating set of $G$ and every vertex in $D$ is within distance $2$ from  another vertex of $D$.
The semitotal domination number of $G$ is the minimum cardinality of a semitotal dominating set of $G$ and is denoted by $\gamma_{t2}(G)$. In this paper after computation of  semitotal domination number of specific  graphs, we count the number of this kind of dominating sets of arbitrary size in some graphs.  	
\end{abstract}


\medskip
\noindent{\bf AMS Subj.\ Class.:} 05C15, 05C25.

\section{Introduction}

A dominating set of a graph $G=(V,E)$ is any subset $S$ of $V$ such that every vertex not in $S$ is adjacent to at least one member of $S$.  
The minimum cardinality of all dominating sets of $G$ is called  the  domination number of $G$ and is denoted by $\gamma(G)$. This parameter has  been extensively studied in the literature and there are  hundreds of papers concerned with domination.  
For a detailed treatment of domination theory, the reader is referred to \cite{domination}. Also, the concept of domination and related invariants have
been generalized in many ways. Among the best know generalizations are total, independent, and connected dominating, each of them with the corresponding domination number. Most of the papers published so far deal with structural
aspects of domination, trying to determine exact expressions for $\gamma(G)$  or some upper and/or lower bounds for it. There were no paper concerned with the 
enumerative side of the problem by 2008.  
Regarding to enumerative side of dominating sets,  Alikhani and Peng \cite{saeid1}, have introduced the domination polynomial of a graph. The domination polynomial of graph $G$ is the  generating function for the number of dominating sets of  $G$, i.e., $D(G,x)=\sum_{ i=1}^{|V(G)|} d(G,i) x^{i}$ (see \cite{euro,saeid1}).   This  polynomial and its roots has been actively studied in recent
years (see for example \cite{Filomat}). 

It is natural to count the number of another kind of dominating sets (\cite{utilitas,weakly}).    Motivated by these papers, we consider another type of dominating set of a graph in this paper.  

A total dominating set, abbreviated a TD-set, of a graph $G$ with no isolated
vertex is a set $D$ of vertices of $G$ such that every vertex in $V(G)$ is adjacent to at
least one vertex in $D$. The total domination number of $G$, denoted by $\gamma_t(G)$, is
the minimum cardinality of a TD-set of $G$. Total domination is now well studied
in graph theory. The literature on the subject of total domination in graphs has
been surveyed and detailed in a  book.  
A set $D$ of vertices in an isolate-free graph $G$ is a semitotal dominating set of $G$ if $D$ is a dominating set of $G$ and every vertex in $D$ is within distance $2$ from  another vertex of $D$.
The semitotal domination was introduced by Goddard, Henning and McPillan \cite{Goddard}, and studied further in \cite{semi1,semi2} and elsewhere.

The semitotal domination number of $G$ is the minimum cardinality of a semitotal dominating set of $G$ and is denoted by $\gamma_{t2}(G)$. By the definition it is easy  to see that for any graph $G$ with no isolated vertices, $\gamma(G)\leq \gamma_{t2}(G)\leq \gamma_t(G)$. Straight from the definition we see that $\gamma_{t2}(G)\geq 2$ but in this paper we consider $\gamma_{t2}(K_n)=1$. 
Recently, Henning, Pal and Pradhan \cite{DMGT} studied the semitotal domination number in block graphs. 
They presented a linear time algorithm to
compute a minimum semitotal dominating set in block graphs. Also they studied the complexity of the semitotal domination problem. 
A domination-critical (domination-super critical, respectively) vertex in a graph
$G$ is a vertex whose removal decreases (increases, respectively) the domination
number. Bauer et al. \cite{Bauer}  introduced the concept of domination stability in graphs.
The domination stability, or just $\gamma$-stability, of a graph $G$ is the minimum number
of vertices whose removal changes the domination number. Motivated by domination stability, we consider  the semi-total stability of a graph.  %Also we  study  the edge semi-total  stability number (semi-total bondage number) of $G$ and compute it for some specific graphs.

In Section 2, we compute the semitotal domination number of specific  graphs. In Section 3, we count the number of semitotal  dominating sets of arbitrary size in some graphs.  Finally in Section 4, we introduce  semitotal domination stability of a graph and compute it for some graphs.

\section{Semitotal domination number of specific graphs} 

In this section, we study  the semitotal domination number of some specific graphs. Here, we recall some graph products. The {\it corona product} $G\circ H$ of two graphs $G$ and $H$ is defined as the graph obtained by taking one copy of $G$ and $\vert V(G)\vert $ copies of $H$ and joining the $i$-th vertex of $G$ to every vertex in the $i$-th copy of $H$. 
The Cartesian product of graphs $G$ and $H$ is a graph denoted $G\Box H$ whose
vertex set is $V (G) \times  V (H)$. Two vertices $(g, h)$ and $(g', h')$ are adjacent if
either $g = g'$ and $hh'\in  E(H)$, or $gg' \in  E(G)$ and $h = h'$.
The {\it join} of two graphs $G_1$ and $G_2$, denoted by $G_1\vee G_2$,
is a graph with vertex set  $V(G_1)\cup V(G_2)$
and edge set $E(G_1)\cup E(G_2)\cup \{uv| u\in V(G_1)$ and $v\in V(G_2)\}$. 
We begin with computation of the semi-total domination number of specific graphs which is straightforward to compute.  

\begin{theorem} \label{Thm1}
	\begin{enumerate}
		\item[(i)]  
		For every $n\geq 3$, $\gamma_{t2}(P_{n})=\gamma_{t2}(C_{n})=\lceil \frac{2n}{5}\rceil$. 
		
		\item[(ii)]  If $W_n$ is a wheel of order $n$, then $\gamma_{t2}(W_n)=\lceil \frac{n-1}{3}\rceil$. 
		\item[(iii)] 
		If $F_n$ is a friendship graph (join of $K_1$ and $nK_2$), then $\gamma_{t2}(F_n)=n.$
		
		\item[(iv)] If $B_n$ is a book graph (the Cartesian product $K_{1,n}\square P_2$), then
		$\gamma_{t2}(B_n)=n+1$.
		\item[(v)]
		$		\gamma_{t2}(K_{m,n})=\left\{
		\begin{array}{ll}
		min\{m,n\}& 2\leq n,m\leq 4\\
		4& m,n\geq 5
		\end{array}\right.$ 
	\end{enumerate} 
\end{theorem}

The following theorem gives the difference between $\gamma(G)$ and $\gamma_{t_2}(G)$ for certain graphs, which are easy to prove.

	\begin{theorem} \label{Thm2.2}
		\begin{enumerate}
			\item[(i)] 	\begin{equation*}
			\gamma_{t2}(P_n)- \gamma(P_n)=\left\{
			\begin{array}{ll}
			0 &n=4,5,7,10\\
			1 &6\leqslant n \leqslant 22 ,\quad n\neq 7,10,21,18\\
			2 &23\leqslant n\leqslant 37,\quad n\neq 25,33,36\\
			3 &38\leq n\leq 52,\quad n\neq 40,48,51\\
			4 & 53\leq n\leq 67,\quad n\neq 55,63,66\\
			5 & 68\leq n\leq 82,\quad n\neq 70,78,81\\
			\geq 6 & n\geq83,\quad n\neq 85
			\end{array}\right. 
			\end{equation*}
								
			\item[(ii)]  
			For the Petersen graph $P$,  $\gamma_{t2}(P)=\gamma(P)$.
			
			\item[(iii)] 			$\gamma_{t2}(B_n)=\gamma(B_n)+n-1.$
			
			\item[(iv)]  $\gamma_{t2}(F_n)-\gamma(F_n)=n-1.$
			\item[(v)]  For the star graph $S_n=K_{1,n}$,  
			$\gamma_{t2}(S_n)-\gamma(S_n)=n-1.$ 
			\item[(vi)] 	$\gamma_{t2}(W_n)-\gamma(W_n)=\lceil \frac{n-1}{3}\rceil-1.$ 
				\item[(vii)] 
				For the complete bipartite graph $K_{m,n}$ with $m\leq n$, 
				\begin{equation*}
				\gamma_{t2}(K_{m,n})-\gamma(K_{m,n})=\left\{
				\begin{array}{ll}
				m-2& 2\leq m\leq 4\\
				2& m\geq 5
				\end{array}\right. 
				\end{equation*}
				\end{enumerate} 
	\end{theorem}

		The following theorem is about the semitotal domination number of corona and join products of two graphs. 
	
	\begin{theorem} 
	\begin{enumerate}
	\item[(i)]  
	If $G_1$ and $G_2$ are two graphs, then 
	$$\gamma_{t2}(G_1\circ  G_2)\leq\gamma_{t2}(G_1)+\gamma_{t2}(G_2)\times(|V(G_1)|-\gamma_{t2}(G_1)).$$ 
	Moreover, this inequality is sharp, when $G_2$ is a complete graph. 
	
	\item[(ii)]  For two graphs $G$ and $H$ (which are not complete graphs) of order at least three, 
	$$ \gamma_{t2}(G\vee H)=min\{\gamma_{t2}(G), \gamma_{t2}(H),4\}$$  
	\end{enumerate}
	\end{theorem}

	\proof
	\begin{enumerate}
	\item[(i)]   By the construction of $G_1\circ G_2$, the vertices in the semitotal dominating set of $G_1$
	covers all the verticies of copies of $G_2$ which adjacent to them. Suppose that $D$ is a semitotal dominating set of $G_1$. Every vertex in $V(G_1)\setminus D$ adjacent to one copy of $G_2$, and  therefore, these  vertices cover by the semitotal dominating set of $G_2$. So
	 $\gamma_{t2}(G_1\circ G_2)\leq \gamma_{t2}(G_1)+\gamma_{t2}(G_2)\times(|V(G_1)|-\gamma_{t2}(G_1)).$
	
	\item[(ii)]     By the construction of $G\vee H$, all of the vertices of $G$ are adjacent to all of the vertices of $H$, and therefore any semitotal dominating set of $G$ is a semitotal dominating set of $G\vee H$ and also any semitotal dominating set of $H$ is a semitotal dominating set of $G\vee H$. If $\gamma_{t2}(G)>4$, $\gamma_{t2}(H)>4$,
	then we can cover all vertices of $G\vee H$ by four vertices. So we have the result.\qed
	\end{enumerate}
	\begin{theorem}
	If $G$ is a complete graph and $H$ is an arbitrary graph (which is not complete graph), then 
	$$ \gamma_{t2}(G\vee H)=\gamma_{t2}(H).$$
	\end{theorem}
		
		\begin{proof}
		Since in $G\vee H$ all vertices of $G$ are adjacent to all vertices of $H$, so the semitotal dominating set of $H$ is a semitotal dominating set of $G\vee H$. On the other hand, in $G$ the distance between two vertices is equal one and so  we have $ \gamma_{t2}(G\vee H)=\gamma_{t2}(H)$.\qed
		\end{proof}

	\begin{theorem} 
		$\gamma_{t2}(P_n\Box P_m) =\lceil \frac{2n}{5}\rceil\times \lceil \frac{m}{3}\rceil+ \lfloor\frac{m}{3}\rfloor \times (n-\lceil \frac{2n}{5}\rceil).$
	\end{theorem}
	\begin{proof}
	By the construction of $P_n\Box P_m$ we have $m$ copies of $P_n$. The vertices of the second copy that adjacent to the semitotal dominating set of the first copy cover by these vertices and we can cover other 
	vertices of the second  copy by the complement of the semitotal dominating set of the third  copy. By continuing this method for  other copies, the number of at least vertices  that we can cover all vertices of $P_n\Box P_m$ by them, is  
$\lceil \frac{2n}{5}\rceil\times \lceil \frac{m}{3}\rceil+ \lfloor\frac{m}{3}\rfloor \times (n-\lceil \frac{2n}{5}\rceil).$\qed
	\end{proof}

		%\begin{theorem}
%	If $G$ is a disconnected graph with components $G_1,G_2,...,G_n$, then
	% $$\gamma_{t2}(G)=\gamma_{t2}(G_1)+ \gamma_{t2}(G_2)+...+\gamma_{t2}(G_n).$$
%	\end{theorem}

\section{The number of semitotal dominating sets} 

 In this section, we consider the problem of the number of the semitotal dominating sets of any size in a graph $G$. 
	Let ${\mathcal D}_{t2}(G,i)$ be the family of
	semitotal   dominating sets of a graph $G$ with cardinality $i$ and let
	$d_{t2}(G,i)=|{\mathcal D}_{t2}(G,i)|$. We denote the generating function for the number of semitotal dominating sets of $G$ by $D_{t2}(G,x)$ and is the polynomial
	$$D_{t2}(G,x)=\sum_{ i=1}^{|V(G)|} d_{t2}(G,i) x^{i},$$ and we call it semitotal domination polynomial of $G$.  
	Here, we try to count the number of this kind of dominating sets and study the semitotal domination polynomial for certain graphs.

	\begin{theorem} 
	\begin{enumerate} 
	\item[(i)] For every $i\neq n$, $d_{t2}(K_{1,n},i)=0$, $d_{t2}(K_{1,n},n)=1.$
	
	\item[(ii)] For every $n\geq 3$, $D_{t2}(K_{1,n},x)=x^n$.  
	\end{enumerate} 
	\end{theorem} 
	\begin{proof}
		\begin{enumerate} 
			\item[(i)]
	Since $\gamma_{t2}(K_{1,n})=n$ so for $i< n, d_{t2}(K_{1,n},i)=0$ and since  the distance between central vertex with other vertices is equal one, so this vertex is not in the semitotal dominating set of $K_{1,n}$, and so $d_{t2}(K_{1,n},n)=1.$
	\item[(ii)]
	Follows from  Part (i) and the definition of the semitotal domination polynomial. \qed	
\end{enumerate} 
	\end{proof}

	\begin{theorem} For a bipartite graph $K_{m,n}$ with $3 \geq m$ and $m \leq n$, we have 
	\begin{equation*}
	d_{t2}(K_{m,n},i)=\left\{
	\begin{array}{ll}
	0  & i \leq m-1\\
	{m+n\choose m}-{n\choose m}-m{n\choose m-1}  & i=m\\
	{m+n\choose i}-{n\choose i}-m{n\choose i-1}-n{m\choose i-1} &i>m ,i\neq n\\
	{m+n\choose n}-mn-n{m\choose n-1}&  i=n
	\end{array}\right. 
	\end{equation*}
	\end{theorem}
	\begin{proof}
	If $m\leq 3, \gamma_{t2}(K_{m,n})=m$ and so $d_{t2}(K_{m,n},i)=0$ for $i\leq m-1$. If $i\geq m$,  the number of sets with  $i$ vertices is  ${m+n\choose i}$, but since   the distance between one vertex of the first section and one vertex of the second section is one, so some of these sets  are not semitotal dominating set. The number of  $i$-sets which cannot be semitotal dominating sets, are  ${m\choose 1}\times {n\choose i-1}$, ${n\choose 1}\times {m \choose i-1}$ and  also for $i\neq n$,  ${n \choose i}$. So we have the result.\qed
\end{proof}

	Similarly, we have the following theorem: 
		
	\begin{theorem} For a bipartite graph $K_{m,n}$ with $4\leq m\leq n$, we have 
	\begin{equation*}
	d_{t2}(K_{m,n},i)=\left\{
	\begin{array}{ll}
	0  & i \leq 3\\
	{m+n\choose m}-{n\choose m}-m{n\choose m-1}  & i=m\\
	{m+n\choose i}-{n\choose i}-m{n\choose i-1}-n{m\choose i-1} &i>3 ,i\neq n\\
	{m+n\choose n}-mn-n{m\choose n-1}&  i=n
	\end{array}\right. 
	\end{equation*}
	\end{theorem}

	%\begin{corollary}
	%If $4<m<n$ then $d_{t2}(K_{m,n},4)=C(m,2)\cdot C(n,2)$ and\\ $d_{t2}(K_{m,n},5)=C(m,2)\cdot C(n,3)+C(m,3)\cdot C(n,2)$.
	%\end{corollary}
	\begin{theorem} 
	\begin{enumerate} 
	\item[(i)] For every $i\geq n\geq 2$, $	d_{t2}(F_n,i)=	2^{n} {n\choose i-n}.$
	
	\item[(ii)] For every $n\geq 2$, $D_{t2}(F_n,x)=2^nx^n(1+x)^n$.  
	\end{enumerate} 
	\end{theorem}
	\begin{proof}
	\begin{enumerate}
	\item[(i)] Since $\gamma_{t2}(F_n)=n$ so if $i<n$ , $d_{t2}(F_n,i)=0$.
	If $i\geq n$, then first we should select  $n$ vertex from sides of $n$ triangles by $2^{n}$ methods  
	and then  select $i-n$ vertex with ${n\choose i-n}$ ways. Therefore $d_{t2}(F_n,i)=	2^{n} {n\choose i-n}$.
	\item[(ii)]  
	It follows by Part (i) and definition of semitotal domination polynomial.\qed
	\end{enumerate}
\end{proof}

	We need the following lemma to obtain more results:
	
	\begin{lemma}{\rm \cite{Goddard}}
	 If $G$ is a connected graph on $n\geq 4$ vertices, then $\gamma_{t_2}(G)\leq \frac{n}{2}$. 
	 	\end{lemma}
	
	Goddard, Henning, and  McPillan in \cite{Goddard} characterized the trees with semitotal domination number exactly one-half order. They defined a family $\mathcal{T}$ of trees as follows. Let $H$ be a nontrivial tree and for each vertex $v$ of $H$, add either a $P_2$ or a $P_4$ and identify $v$ with one end vertex of the path. They proved the following theorem: 
	
	\begin{theorem} \label{half}
		Let $T$ be a tree of order $n\geq 4$. Then $\gamma_{t_2}(T) =\frac{n}{2}$ if and only if $T\in \mathcal{T} $ or $T=K_{1,3}$.
	\end{theorem}
	
	Now, we state and prove the following result: 
	
	\begin{theorem}
	For every tree $T\in \mathcal{T}$, $D_{t_2}(T,x)=D(T,x)$. 
	\end{theorem} 
	\begin{proof}
	We should prove that for every $i\geq \gamma_{t_2}(T)$, $d_{t_2}(T,i)=d(T,i)$. Suppose that $i\geq \gamma_{t_2}(T)$, by Lemma \ref{half}, every dominating set of $T$ with cardinality $i\geq \gamma_{t_2}(T)$ is a semitotal dominating set of $T$. Therefore, we have the result. \qed 	
		\end{proof}

			Goddard, Henning, and  McPillan in \cite{Goddard} extended Theorem \ref{half} from trees to  all graphs. For given graphs $G$ and $H$ and  every vertex in $G$, form a copy of $H$ and identify one vertex in the copy of $H$ with the corresponding vertex in $G$. Let to denote this as $G\diamond H$ (See $P_5\diamond C_4$ in Figure \ref{diamond}). The following theorem  characterize the graphs with with minimum degree at least $2$ whose  semitotal domination number is exactly one-half order.
			
				\begin{theorem} \label{halfgraph}{\rm \cite{Goddard}}
					Let $G$ be a connected graph  of order $n\geq 4$ with minimum degree at least $2$. Then $\gamma_{t_2}(T) =\frac{n}{2}$ if and only if $G$ is $C_6, C_8$, a spanning subgraph of $K_4$ or $H\diamond C_4$ for some graph $H$. 
				\end{theorem}

				Now, we have  the following result: 
				
				\begin{theorem}
					 $D_{t_2}(H\diamond C_4,x)=D(H\diamond C_4,x)$. 
				\end{theorem} 
				\begin{proof}
					We should prove that for every $i\geq \gamma_{t_2}(H\diamond C_4)$, $d_{t_2}(H\diamond C_4,i)=d(H\diamond C_4,i)$. Suppose that $i\geq \gamma_{t_2}(H\diamond C_4)$, by Lemma \ref{halfgraph}, every dominating set of $H\diamond C_4$ with cardinality $i\geq \gamma_{t_2}(H\diamond C_4)$ is a semitotal dominating set of $H\diamond C_4$. Therefore, we have the result. \qed 	
				\end{proof} 
			
			\begin{figure}
				\begin{center}
					\includegraphics[width=0.7\textwidth]{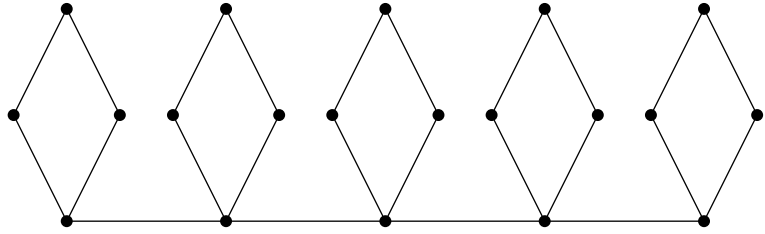}
					\caption{ \label{diamond} The graph  $P_5\diamond C_4$. }
				\end{center}
			\end{figure}
	
A split graph is a graph in which the vertices can be partitioned into a clique and an independent set. Figure \ref{split} shows a split graph partitioned into a clique (induced graph by $\{1,2,3\}$) and an independent set induced graph by $\{4,5\}$. 

\begin{theorem} 
	If $G$ is a  connected split graph $G$ with no dominating vertex, then 
	\[
	D_t(G,x)=D_{t_2}(G,x)=D(G,x).	
	\]
\end{theorem} 

\begin{proof} 
First we show that   $\gamma(G)=\gamma_{t_2}=\gamma_t(G)$. It is suffices to prove that $\gamma_t(G)\leq \gamma(G)$. Suppose that $V(G)=C\cup I$ is a partition of the vertices of $G$ into a clique $C$ and an independent set $I$. Consider a  minimum dominating set of $G$ contained in $C$ such as $D$. If $D$  contains $v\in I$, then since no neighbour of $v$ such as $u$ is in $D$, $(D\setminus \{v\})\cup \{u\}$  is a minimum dominating set containing less vertices of $I$. Since $G$ has no dominating vertex, every
dominating set contained in $C$ is a total dominating set and so $\gamma_t(G) \leq \gamma(G)$. So every semitotal dominating set of cardinality $i$ of $G$ is a total dominating set of cardinality $i$ and is a dominating set of $G$ with cardinality $i$. Therefore $d(G,i)=d_t(G,i)=d_{t2}(G,i)$ and so we have the result. \qed
\end{proof}

		\begin{figure}
			\begin{center}
				\includegraphics[width=0.5\textwidth]{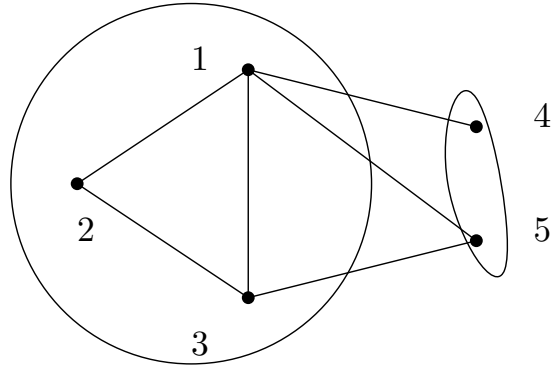}
				\caption{ \label{split} Example of split graph. }
			\end{center}
		\end{figure}
	
	\section{Stability of semitoal domination number}
	
	In this section, we introduce the semitotal domination stability of a graph and  compute this parameter for some specific graphs.
	\begin{definition}
		Let $G$ be a graph of order $n\geq2$. The stabilizing on the semitotal domination  number, or just semitotal,  $st_{\gamma_{t2}}(G)$ of graph $G$ is the  minimum number of vertices whose
		removal changes the semitotal domination   number.
	\end{definition} 
	
	\begin{theorem}
		If $m\leq n$, then
	\begin{equation*}
	st_{\gamma_{t2}}(K_{m,n})=\left\{
	\begin{array}{ll}
	0, & m=2\\
	1,  &3\leq m \leq 4\\
	m-3,  &m > 4\\
	\end{array}\right. 
	\end{equation*}
	\end{theorem} 
	\proof
	Suppose that $3\leq m\leq 4$ and $m\leq n$. In this case $\gamma_{t2}(K_{m,n})=\min\{m,n\}=m$ and so removing one vertex changes $\gamma_{t2}(K_{m,n})$. Therefore in this case,
	$st_{\gamma_{t2}}(K_{m,n})=1$.    
	Suppose  that $4<m\leq n$, in this case $\gamma_{t2}(K_{m,n})=4$, so the number of minimum vertex whose removal changes the $\gamma_{t2}(K_{m,n})$  is  $m-3$.\qed

	\begin{theorem} 
	\begin{equation*}
	st_{\gamma_{t2}}(P_{n})=st_{\gamma_{t2}}(C_{n})=\left\{
	\begin{array}{ll}
	1, &n=5k+1,\quad  n=5k+3\\
	2, &n=5k+2, \quad  n=5k-1\\
	3, &n=5k.  
	\end{array}\right. 
	\end{equation*}
	\end{theorem}
	\proof
	We know that $\gamma_{t2}(P_n)=\lceil \frac{2n}{5}\rceil$ (Theorem \ref{Thm1}). For $n=5k+1$ and $n=5k+3$, $\gamma_{t2}(P_{n-1})=\gamma_{t2}(P_{n})-1$. So in this case $st_{\gamma_{t2}}(P_n)=1$.
	With similar arguments we have the results for another cases. \qed

	\begin{theorem}  
	\begin{equation*}
	st_{\gamma_{t2}}(W_{n})=\left\{
	\begin{array}{ll}
	1, &n=3k+2 \\
	2, &n=3k \\
	3, & n=3k+1 \\
	\end{array}\right. 
	\end{equation*}
	\end{theorem} 
	\proof
We know that $\gamma_{t2}(W_n)=\lceil\frac{n-1}{3}\rceil$ (Theorem \ref{Thm1}). Since 
	$\lceil\frac{3k+2-1}{3}\rceil=\lceil\frac{3k+1-1}{3}\rceil+1$
	so for the case $n=3k+2$, $st_{\gamma_{t2}}(W_{n})=1$. 	With similar arguments we have the results for another cases. \qed 
	
	\begin{theorem} 
	\begin{enumerate}
	\item[(i)]   If $5\leq n \leq 10$ and $m>n$ then
	\begin{equation*}
	st_{\gamma_{t2}}(P_{n}\vee P_{m})=\left\{
	\begin{array}{ll}
	1 &n=5k+1,\quad n=5k+3\\
	2 &n=5k+2 ,\quad n=5k-1\\
	3 &n=5k \\
	\end{array}\right. 
	\end{equation*}
	\item[(ii)]  If $n>10$ and $n\leq m$ then $st_{\gamma_{t2}}(P_{n}\vee P_{m})=n-7$.
	\end{enumerate}
	\end{theorem} 
	\begin{proof}
			\begin{enumerate}
				\item[(i)] 
	Suppose that $5\leq n \leq 10$ and $m>n$. Since $\gamma_{t2}(P_{n}\vee P_{m})=\gamma_{t2}(P_n)$ 
	so in this case $st_{\gamma_{t2}}(P_{n}\vee P_{m})=st_{\gamma_{t2}}(P_{n})$.
		\item[(ii)]
			If $n>10$ by  Theorem \ref{Thm2.2}, $\gamma_{t2}(P_{n}\vee P_{m})=4$ and by attention to $\gamma_{t2}(P_7)=\lceil\dfrac{2\times 7}{5}\rceil=3$ we  conclude $st_{\gamma_{t2}}(P_{n}\vee P_{m})=n-7$. \qed
			\end{enumerate} 
	\end{proof}
	\begin{theorem}
	$st_{\gamma_{t2}}(P_{n}\square P_{m})=\lceil\frac{2n}{5}\rceil.$
	\end{theorem}

	\begin{theorem} 
	\begin{enumerate}
	\item[(i)]  
	If $F_n$ is a friendship graph, then $st_{\gamma_{t2}}(F_n)=2.$
	
	\item[(ii)] If $B_n$ is a book graph (the Cartesian product $K_{1,n}\square P_2$), then
	$st_{\gamma_{t2}}(B_n)=1$.
	\item[(iii)] If $S_n$ is a star graph then
	$st_{\gamma_{t2}}(S_n)=1$
	\end{enumerate} 
	\end{theorem}
	\begin{proof}
	\begin{enumerate}
	\item[(i)]   Since $\gamma_{t2}(F_n)=n$, so $\gamma_{t2}(F_{n-1})=n-1$ and obviously to reach from $F_n$ to $F_{n-1}$, we need to remove two vertices. So $st_{\gamma_{t2}}(F_n)=2.$
	\item[(ii)]  Since $\gamma_{t2}(B_n)=n+1$, so $\gamma_{t2}(B_{n-1})=n$ and obviously to reach from $B_n$ to $B_{n-1}$, we need to remove two vertices. So $st_{\gamma_{t2}}(B_n)=2.$
	\item[(iii)]  Since $\gamma_{t2}(S_n)=n$ and center vertex is not in semitotal dominating set, so by removing one vertex of  $S_n$ we reach to  $S_{n-1}$. Therefore $st_{\gamma_{t2}}(S_n)=1$.\qed
\end{enumerate}
	\end{proof}

\end{document}